\documentclass[oneside,english]{amsart}
\usepackage[T1]{fontenc}
\usepackage[latin9]{inputenc}
\usepackage{geometry}
\usepackage{units}
\usepackage{amstext}
\usepackage{amsthm}
\usepackage{amssymb}

\makeatletter
\numberwithin{equation}{section}
\numberwithin{figure}{section}
\theoremstyle{plain}
\newtheorem{thm}{\protect\theoremname}
\theoremstyle{definition}
\newtheorem{defn}[thm]{\protect\definitionname}
\theoremstyle{plain}
\newtheorem{prop}[thm]{\protect\propositionname}
\theoremstyle{plain}
\newtheorem{lem}[thm]{\protect\lemmaname}

\usepackage{nicefrac}
\usepackage{babel}
\@ifundefined{definecolor}
 {\usepackage{color}}{}
\@ifundefined{definecolor}{\usepackage{color}}{}
\usepackage{graphicx}

\usepackage[mathscr]{euscript}
 \let\mathscr\relax% just so we can load this and rsfs
\usepackage[scr]{rsfso}

\usepackage{mathtools}
\usepackage{tikz}

\makeatother

\usepackage{babel}
\providecommand{\definitionname}{Definition}
\providecommand{\lemmaname}{Lemma}
\providecommand{\propositionname}{Proposition}
\providecommand{\theoremname}{Theorem}

\begin{document}
\title{A proof of the Graham Sloane conjecture.}
\author{Edinah K. Gnang, Michael Peretzian Williams }
\dedicatory{Dedicated to the memory of Ronald Graham.}
\begin{abstract}
We settle in the affirmative the Graham-Sloane conjecture.
\end{abstract}

\maketitle

\section{Introduction.}

A well-known, conjecture in the area of graph labelings concerns the
harmonious labelling. This labelling was introduced by Graham and
Sloane \cite{GS80} and was motivated by the study of additive bases.
Given an Abelian group $\Gamma$ and a graph $G$, we say that a labelling
$L$ : $V(G)\rightarrow\Gamma$ is $\Gamma$--harmonious if the map
$L^{\prime}$ :$E(G)\rightarrow\Gamma$ defined by $L^{\prime}\left(u,v\right)=L\left(u\right)+L\left(v\right)$
is injective. In the case when $\Gamma$ is the group of integers
modulo $n$ we omit it from our notation and simply call such a labelling
harmonious. The Graham--Sloane \cite{GS80} conjecture, better known
as the \emph{Harmonious Labeling Conjecture }(HLC), asserts that every
tree admits a harmonious labeling. For a detail survey of the extensive
literature on graph labeling problems, see \cite{gallsurvey}. Recently
Montgomery, Pokrovskiy and Sudakov showed in \cite{MPS} that that
every tree is almost harmonious. More precisely they show that every
$n$--vertex tree $T$ has an injective $\Gamma$-harmonious labelling
for any Abelian group $\Gamma$ of order $n+o\left(n\right)$. In
the present note we view harmonious labeling of graphs as special
vertex labelings which results in a bijection between vertex labels
and \emph{induced additive edge labels}. Induced additive edge labels
correspond to the residue classes modulo $n$ of the sum of integers
assigned to the vertices spanning each edge. Our discussion is based
upon a functional reformulation of the HLC. A rooted tree on $n>0$
vertices is associated with a function 
\begin{equation}
f\in\left(\nicefrac{\mathbb{Z}}{n\mathbb{Z}}\right)^{\nicefrac{\mathbb{Z}}{n\mathbb{Z}}}\text{ subject to }\left|f^{(n-1)}\left(\nicefrac{\mathbb{Z}}{n\mathbb{Z}}\right)\right|=1,\label{functional_reformulation}
\end{equation}
\[
\text{where}
\]
\[
\forall\,i\in\mathbb{Z}_{n},\;f^{(0)}\left(i\right)\,:=i,\mbox{ and }\forall\,k\ge0,\;f^{(k+1)}\left(i\right)=f^{(k)}\left(f\left(i\right)\right)=f\left(f^{(k)}\left(i\right)\right).
\]
In other words the function $f$ has a unique fixed point which is
attractive over the whole domain of $f$. Every $f\in\left(\nicefrac{\mathbb{Z}}{n\mathbb{Z}}\right)^{\nicefrac{\mathbb{Z}}{n\mathbb{Z}}}$
has a corresponding \emph{functional directed graph }denoted \emph{$G_{f}$}
whose vertex, edge sets and automorphism group are respectively
\[
V\left(G_{f}\right):=\nicefrac{\mathbb{Z}}{n\mathbb{Z}},\quad E\left(G_{f}\right):=\left\{ \left(i,f\left(i\right)\right)\,:\,i\in\nicefrac{\mathbb{Z}}{n\mathbb{Z}}\right\} \quad\text{ and }\quad\text{Aut}\left(G_{f}\right).
\]
When $f$ is subject to Eq. (\ref{functional_reformulation}), the
corresponding functional directed graph $G_{f}$ is a directed rooted
tree with an additional loop edge placed at its fixed point. The edges
of $G_{f}$ are oriented to ensure that every vertex has out-degree
one. In other words each edges of the rooted tree is oriented to point
towards the root of the tree (i.e. the fixed point). Applying the
swap sink transformation to $G_{f}$ subject to Eq. (\ref{functional_reformulation})
results in a rooted directed tree $G_{S\left(f,k\right)}$ associated
with the function S$\left(f,k\right)\in\left(\nicefrac{\mathbb{Z}}{n\mathbb{Z}}\right)^{\nicefrac{\mathbb{Z}}{n\mathbb{Z}}}$.
In other words the graph $G_{S\left(f,k\right)}$ differs from $G_{f}$
in the fact that its loop is relocated to a new vertex labeled $k$
and some edges are re-oriented to ensure that every vertex has out-degree
one. For instance when $f$ denotes the identically zero function
\[
E\left(G_{f}\right)=\left\{ \left(i,0\right)\,:\,i\in\nicefrac{\mathbb{Z}}{n\mathbb{Z}}\right\} \implies E\left(G_{S\left(f,1\right)}\right)=\left\{ \left(0,\,1\right),\,\left(1,\,1\right)\right\} \cup\left\{ \left(i,\,0\right)\,:\,i\in\mathbb{Z}_{n}-\left\{ 0,1\right\} \right\} .
\]
Examples of induced edge labels associated with a functional directed
graph $G_{f}$ of $f\in\left(\nicefrac{\mathbb{Z}}{n\mathbb{Z}}\right)^{\nicefrac{\mathbb{Z}}{n\mathbb{Z}}}$
include : 
\begin{itemize}
\item Induced additive edge labels given by $\left\{ f\left(i\right)+i\mod n:i\in\mathbb{Z}_{n}\right\} $
and the graph $G_{f}$ is harmonious if there exists $\sigma\in\text{S}_{n}\subset\left(\nicefrac{\mathbb{Z}}{n\mathbb{Z}}\right)^{\nicefrac{\mathbb{Z}}{n\mathbb{Z}}}$
such that
\[
\mathbb{Z}_{n}=\left\{ \sigma f\sigma^{\left(-1\right)}\left(i\right)+i\mod n:i\in\nicefrac{\mathbb{Z}}{n\mathbb{Z}}\right\} .
\]
\item More general $\tau$-induced edge labels given by $\left\{ \tau\left(f\left(i\right),i\right)\,:\,i\in\mathbb{Z}_{n}\right\} $,
for some $\tau\in\left(\nicefrac{\mathbb{Z}}{n\mathbb{Z}}\right)^{\nicefrac{\mathbb{Z}}{n\mathbb{Z}}\times\nicefrac{\mathbb{Z}}{n\mathbb{Z}}}$
and $G_{f}$ is $\tau$-Zen for if there exists $\sigma\in\text{S}_{n}\subset\left(\nicefrac{\mathbb{Z}}{n\mathbb{Z}}\right)^{\nicefrac{\mathbb{Z}}{n\mathbb{Z}}}$
such that
\[
\mathbb{Z}_{n}=\left\{ \tau\left(\sigma f\sigma^{\left(-1\right)}\left(i\right),i\right)\,:i\in\nicefrac{\mathbb{Z}}{n\mathbb{Z}}\right\} .
\]
\end{itemize}
Our main result is a proof that when $n$ is odd, for all $f\in\left(\nicefrac{\mathbb{Z}}{n\mathbb{Z}}\right)^{\nicefrac{\mathbb{Z}}{n\mathbb{Z}}}$
subject to $\left|f^{(n-1)}\left(\nicefrac{\mathbb{Z}}{n\mathbb{Z}}\right)\right|=1$,
there exist $k\in\nicefrac{\mathbb{Z}}{n\mathbb{Z}}$ such that
\[
n=\max_{\sigma\in\text{S}_{n}}\left|\left\{ \sigma\text{S}\left(f,k\right)\sigma^{(-1)}\left(i\right)+i:i\in\nicefrac{\mathbb{Z}}{n\mathbb{Z}}\right\} \right|.
\]
Which settles in the affirmative the Graham-Sloane conjecture.
\begin{defn}
Let 
\[
f\in\mathbb{Z}_{n}^{\mathbb{Z}_{n}}\text{be subject to }\left|f^{(n-1)}\left(\mathbb{Z}_{n}\right)\right|=1,
\]
then the set HaL$\left(G_{f}\right)$ denotes the subset of distinct
functional directed graphs isomorphic to $G_{f}$ whose labeling is
harmonious. Formally we write 
\[
\text{HaL}\left(G_{f}\right)\,:=\left\{ G_{\sigma f\sigma^{\left(-1\right)}}:\sigma\in\nicefrac{\text{S}_{n}}{\text{Aut}\left(G_{f}\right)}\;\text{ and }\;n=\left|\left\{ \sigma f\sigma^{\left(-1\right)}\left(j\right)+j:\,j\in\nicefrac{\mathbb{Z}}{n\mathbb{Z}}\right\} \right|\right\} .
\]
\end{defn}

\section{The Harmonious Invariance Group}

The following expresses a necessary and sufficient condition for a
functional directed graph to be harmonious.
\begin{prop}
(Harmonious Expansion) Let $G_{f}$ denote the functional directed
graph of $f\in\left(\nicefrac{\mathbb{Z}}{n\mathbb{Z}}\right)^{\nicefrac{\mathbb{Z}}{n\mathbb{Z}}}$.
Then $G_{f}$ is harmonious iff
\begin{equation}
\exists\,\gamma,\sigma_{\gamma}\in\text{S}_{n}\text{ such that }f\left(i\right)=\sigma_{\gamma}^{\left(-1\right)}\left(\gamma\sigma_{\gamma}\left(i\right)-\sigma_{\gamma}\left(i\right)\right),\quad\forall\:i\in\nicefrac{\mathbb{Z}}{n\mathbb{Z}}.\label{Harmonious Expansion}
\end{equation}
The subscript $\gamma$ notation for $\sigma_{\gamma}$, is meant
to emphasize the dependence of the coset representative on the permutation
parameter $\gamma$. 
\end{prop}

\begin{proof}
We prove only the forward direction for the converse easily follows.
Recall that $G_{f}$ is harmonious if there exist $\gamma\in$ S$_{n}$
and $\sigma_{\gamma}\in\nicefrac{\text{S}_{n}}{\text{Aut}\left(G_{f}\right)}$
such that 
\[
\begin{array}{cccc}
 & \sigma_{\gamma}f\sigma_{\gamma}^{\left(-1\right)}\left(i\right)+i & = & \gamma\left(i\right)\\
\\
\implies & \sigma_{\gamma}f\sigma_{\gamma}^{\left(-1\right)}\left(i\right) & = & \gamma\left(i\right)-i\\
\\
\implies & f\left(i\right) & = & \sigma_{\gamma}^{\left(-1\right)}\left(\gamma\sigma_{\gamma}\left(i\right)-\sigma_{\gamma}\left(i\right)\right)
\end{array}\forall\:i\in\nicefrac{\mathbb{Z}}{n\mathbb{Z}},
\]
\end{proof}
\begin{prop}
(Swap Sink Harmony) Let GCD$\left(n,2\right)=1$, let $g\in\left(\nicefrac{\mathbb{Z}}{n\mathbb{Z}}\right)^{\nicefrac{\mathbb{Z}}{n\mathbb{Z}}}$
subject to $\left|g^{(n-1)}\left(\nicefrac{\mathbb{Z}}{n\mathbb{Z}}\right)\right|=1$
and
\[
n-1=\max_{\sigma\in\text{S}_{n}}\left|\left\{ \sigma g\sigma^{\left(-1\right)}+i:i\in\left(\nicefrac{\mathbb{Z}}{n\mathbb{Z}}\right)\backslash g^{\left(n-1\right)}\left(\nicefrac{\mathbb{Z}}{n\mathbb{Z}}\right)\right\} \right|,
\]
then there exist $k\in\nicefrac{\mathbb{Z}}{n\mathbb{Z}}$ such that
\[
n=\max_{\sigma\in\text{S}_{n}}\left|\left\{ \sigma^{\left(-1\right)}\text{S}\left(g,k\right)\sigma\left(i\right)+i:i\in\nicefrac{\mathbb{Z}}{n\mathbb{Z}}\right\} \right|.
\]
\end{prop}

\begin{proof}
The premise that the functional tree $G_{f}$ is subject to the equality
\[
n-1=\max_{\sigma\in\text{S}_{n}}\left|\left\{ \sigma g\sigma^{\left(-1\right)}+i:i\in\left(\nicefrac{\mathbb{Z}}{n\mathbb{Z}}\right)\backslash g^{\left(n-1\right)}\left(\nicefrac{\mathbb{Z}}{n\mathbb{Z}}\right)\right\} \right|,
\]
implies that $G_{g}$ can be relabeled such that its non-loop edges
have distinct additive edge labels. For each such relabeling there
is exactly one congruence class which does not occur as an edge label.
Let us call that label $l$. Since $2$ is not a zero divisor of the
ring $\nicefrac{\mathbb{Z}}{n\mathbb{Z}}$, relocating the loop edge
at the vertex whose label correspond to the solution $x$ to the equation
\[
    2\,x=l
\]
results in a harmoniously labeled graph as claimed. Thereby completing the proof.
\end{proof}
\begin{prop}
(Harmonious Right Invariant Group) Let the graph $G_{g}$ associated
with $g\in\left(\nicefrac{\mathbb{Z}}{n\mathbb{Z}}\right)^{\nicefrac{\mathbb{Z}}{n\mathbb{Z}}}$
be labeled such that for a subset $T\subset\nicefrac{\mathbb{Z}}{n\mathbb{Z}}$
we have 
\[
\left|\left\{ g\left(i\right)+i:i\in\left(\nicefrac{\mathbb{Z}}{n\mathbb{Z}}\right)\backslash T\right\} \right|=n-\left|T\right|.
\]
Let $c$ be an arbitrary element of the ring $\nicefrac{\mathbb{Z}}{n\mathbb{Z}}$
and the graph $G_{g^{\prime}}$ associated with the function
\[
g^{\prime}=g\left(\text{id}+c\right),\ \forall\,i\in\nicefrac{\mathbb{Z}}{n\mathbb{Z}},
\]
then 
\[
\left|\left\{ g^{\prime}\left(i\right)+i:i\in\left(\nicefrac{\mathbb{Z}}{n\mathbb{Z}}\right)\backslash T\right\} \right|=n-\left|T\right|.
\]
\end{prop}

\begin{proof}
By the premise that for some $T\subset\nicefrac{\mathbb{Z}}{n\mathbb{Z}}$
we have 
\[
n-\left|T\right|=\left|\left\{ g\left(i\right)+i:i\in\nicefrac{\mathbb{Z}}{n\mathbb{Z}}-T\right\} \right|.
\]
by the same argument use to prove Prop. (1), there exist $\gamma\in$
S$_{n}$ such that
\[
\begin{array}{cccc}
 & g\left(i\right) & = & \gamma\left(i\right)-i\\
\\
\implies & g\left(i+c\right) & = & \gamma\left(i+c\right)-\left(i+c\right)\\
\\
\implies & \left(g\circ\left(\text{id}+c\right)\right)\left(i\right) & = & \gamma\left(i+c\right)-c-i\\
\\
\implies & g\left(i+c\right) & = & \left(\left(\text{id}+c\right)^{\left(-1\right)}\circ\gamma\circ\left(\text{id}+c\right)\right)\left(i\right)-i
\end{array}\forall\,i\in\nicefrac{\mathbb{Z}}{n\mathbb{Z}}-T.
\]
From which it follows $\left|\left\{ g^{\prime}\left(i\right)+i:i\in\nicefrac{\mathbb{Z}}{n\mathbb{Z}}-T\right\} \right|=\left(n-\left|T\right|\right)$
thereby completes the proof.
\end{proof}
\begin{prop}
(Harmonious Left Invariance Group) Let the graph $G_{g}$ associated
with $g\in\left(\nicefrac{\mathbb{Z}}{n\mathbb{Z}}\right)^{\nicefrac{\mathbb{Z}}{n\mathbb{Z}}}$
be labeled such that for a subset $T\subset\nicefrac{\mathbb{Z}}{n\mathbb{Z}}$
we have 
\[
\left|\left\{ g\left(i\right)+i:i\in\nicefrac{\mathbb{Z}}{n\mathbb{Z}}-T\right\} \right|=n-\left|T\right|.
\]
Let $c$ be an arbitrary element of the ring $\nicefrac{\mathbb{Z}}{n\mathbb{Z}}$
and the graph $G_{g^{\prime\prime}}$ associated with the function
\[
g^{\prime\prime}=g+c,\ \forall\,i\in\nicefrac{\mathbb{Z}}{n\mathbb{Z}},
\]
then 
\[
\left|\left\{ g^{\prime\prime}\left(i\right)+i:i\in\nicefrac{\mathbb{Z}}{n\mathbb{Z}}-T\right\} \right|=\left(n-\left|T\right|\right).
\]
\end{prop}

\begin{proof}
By the premise that for some $T\subset\nicefrac{\mathbb{Z}}{n\mathbb{Z}}$
we have 
\[
n-\left|T\right|=\left|\left\{ g\left(i\right)+i:i\in\nicefrac{\mathbb{Z}}{n\mathbb{Z}}-T\right\} \right|.
\]
by the same argument use to prove Prop. (1), there exist $\gamma\in$
S$_{n}$ such that
\[
\begin{array}{cccc}
 & g\left(i\right) & = & \gamma\left(i\right)-i\\
\\
\implies & g\left(i\right)+c & = & \gamma\left(i\right)+c-i\\
\\
\implies & \left(\left(\text{id}+c\right)\circ g\right)\left(i\right) & = & \left(\text{id}+c\right)\circ\gamma\left(i\right)-i\\
\\
\implies & g\left(i\right)+c & = & \left(\text{id}+c\right)\circ\gamma\left(i\right)-i
\end{array}\;\forall\,i\in\nicefrac{\mathbb{Z}}{n\mathbb{Z}}-T.
\]
From which it follows $\left|\left\{ g^{\prime\prime}\left(i\right)+i:i\in\nicefrac{\mathbb{Z}}{n\mathbb{Z}}-T\right\} \right|=n-\left|T\right|$
thereby completes the proof.
\end{proof}
In Prop. (3) and (4) when $T$ is empty then we say that the $G_{g}$
is harmoniously labeled. It follows as corollary of Prop. (3) and
(4) via an argument similar to the proof of Lagrange's coset theorem
that both the number of harmonious permutations on $n$ vertices and
the number of harmoniously labeled permutations on $n$ vertices,
must be divisible by $n$.

\section{Useful facts about polynomials}

Let $F\left(\mathbf{x}\right),\,G\left(\mathbf{x}\right)\in\mathbb{C}\left[x_{0},\cdots,x_{n-1}\right]$,
be multivariate polynomials which splits into irreducible factors
of the form
\[
F\left(\mathbf{x}\right)=\prod_{0\le i<m}\left(P_{i}\left(\mathbf{x}\right)\right)^{\alpha_{i}},\quad G\left(\mathbf{x}\right)=\prod_{0\le i<m}\left(P_{i}\left(\mathbf{x}\right)\right)^{\beta_{i}},
\]
where $\left\{ \alpha_{i},\beta_{i}\,:\,0\le i<m\right\} \subset\mathbb{Z}_{\ge0}$.
Assume that each factor $P_{i}\left(\mathbf{x}\right)$ is multilinear
in the entries of $\mathbf{x}$ and non-identically constant. Additionally,
$P_{i}\left(\mathbf{x}\right)$ has no common roots in the field of
fractions $\mathbb{C}\left(x_{0},\cdots,x_{k-1},x_{k+1},\cdots,x_{n-1}\right)$,
with any other factor in $\left\{ P_{j}\left(\mathbf{x}\right):0\le j\ne i<m\right\} $
for each $k\in\nicefrac{\mathbb{Z}}{n\mathbb{Z}}$, then 
\[
\begin{array}{c}
\text{LCM}\left(F\left(\mathbf{x}\right),\,G\left(\mathbf{x}\right)\right):=\underset{0\le i<m}{\prod}\left(P_{i}\left(\mathbf{x}\right)\right)^{\max\left(\alpha_{i},\beta_{i}\right)},\\
\text{and}\\
\text{GCD}\left(F\left(\mathbf{x}\right),\,G\left(\mathbf{x}\right)\right):=\underset{0\le i<m}{\prod}\left(P_{i}\left(\mathbf{x}\right)\right)^{\min\left(\alpha_{i},\beta_{i}\right)}.
\end{array}
\]

\begin{defn}
By the quotient remainder theorem an arbitrary $H\left(\mathbf{x}\right)\in\mathbb{C}\left[x_{0},\cdots,x_{n-1}\right]$
admits an expansion of the form
\[
H\left(\mathbf{x}\right)=\sum_{l\in\nicefrac{\mathbb{Z}}{n\mathbb{Z}}}q_{l}\left(\mathbf{x}\right)\,\left(\left(x_{l}\right)^{n}-1\right)+\sum_{g\in\mathbb{Z}_{n}^{\mathbb{Z}_{n}}}H\left(\omega^{g\left(0\right)},\cdots,\omega^{g\left(i\right)},\cdots,\omega^{g\left(n-1\right)}\right)\,\prod_{k\in\nicefrac{\mathbb{Z}}{n\mathbb{Z}}}\left(\prod_{j_{k}\in\nicefrac{\mathbb{Z}}{n\mathbb{Z}}\backslash\left\{ g\left(k\right)\right\} }\left(\frac{x_{k}-\omega^{j_{k}}}{\omega^{g\left(k\right)}-\omega^{j_{k}}}\right)\right),
\]
where $\omega=\exp\left\{ \frac{2\pi\sqrt{-1}}{n}\right\} $. Incidentally,
the canonical representative of the congruence class
\[
H\left(\mathbf{x}\right)\mod\left\{ \begin{array}{c}
\left(x_{k}\right)^{n}-1\\
k\in\nicefrac{\mathbb{Z}}{n\mathbb{Z}}
\end{array}\right\} ,
\]
is defined as the unique polynomial of degree at most $\left(n-1\right)$
in each variable whose evaluations matches exactly evaluations of
$H\left(\mathbf{x}\right)$ over the lattice $\Omega^{n}$ where 
\[
\Omega:=\left\{ \omega^{k}:k\in\nicefrac{\mathbb{Z}}{n\mathbb{Z}}\right\} .
\]
The canonical representative is explicitly expressed as 
\begin{equation}
\sum_{g\in\left(\nicefrac{\mathbb{Z}}{n\mathbb{Z}}\right)^{\nicefrac{\mathbb{Z}}{n\mathbb{Z}}}}H\left(\omega^{g\left(0\right)},\cdots,\omega^{g\left(i\right)},\cdots,\omega^{g\left(n-1\right)}\right)\prod_{k\in\mathbb{Z}_{n}}\left(\prod_{j_{k}\in\nicefrac{\mathbb{Z}}{n\mathbb{Z}}\backslash\left\{ g\left(k\right)\right\} }\left(\frac{x_{k}-\omega^{j_{k}}}{\omega^{g\left(k\right)}-\omega^{j_{k}}}\right)\right),\label{Canonical representative}
\end{equation}
The canonical representative of $H\left(\mathbf{x}\right)$ modulo
algebraic relations $\left\{ \left(x_{i}\right)^{n}-1\,:\,i\in\nicefrac{\mathbb{Z}}{n\mathbb{Z}}\right\} $,
is thus obtained via Lagrange interpolation over the integer lattice
$\Omega^{n}$ as prescribed by Eq. (\ref{Canonical representative}).
Alternatively, the canonical representative is obtained as the final
remainder resulting from performing $n$ Euclidean divisions irrespective
of the order in which distinct divisors are successively taken from
the set $\left\{ \left(x_{i}\right)^{n}-1:i\in\nicefrac{\mathbb{Z}}{n\mathbb{Z}}\right\} $.
\end{defn}

\begin{prop}[Determinantal Certificate]
 For $f\in\left(\nicefrac{\mathbb{Z}}{n\mathbb{Z}}\right)^{\nicefrac{\mathbb{Z}}{n\mathbb{Z}}}$,
subject to $\left|f^{\left(n-1\right)}\left(\nicefrac{\mathbb{Z}}{n\mathbb{Z}}\right)\right|=1$
we have 
\[
n-1=\underset{\sigma\in\text{S}_{n}}{\max}\left|\left\{ \sigma f\sigma^{(-1)}\left(i\right)+i:i\in\left(\nicefrac{\mathbb{Z}}{n\mathbb{Z}}\right)\backslash f^{\left(n-1\right)}\left(\nicefrac{\mathbb{Z}}{n\mathbb{Z}}\right)\right\} \right|
\]
 if and only if
\[
0\not\equiv\text{LCM}\left\{ \prod_{0\le i<j<n}\left(x_{j}-x_{i}\right),\,\prod_{\begin{array}{c}
0\le i<j<n\\
i,j\in\left(\nicefrac{\mathbb{Z}}{n\mathbb{Z}}\right)\backslash f^{\left(n-1\right)}\left(\nicefrac{\mathbb{Z}}{n\mathbb{Z}}\right)
\end{array}}\left(x_{f\left(j\right)}x_{j}-x_{f\left(i\right)}x_{i}\right)\right\} \mod\left\{ \begin{array}{c}
\left(x_{k}\right)^{n}-1\\
k\in\nicefrac{\mathbb{Z}}{n\mathbb{Z}}
\end{array}\right\} 
\]
\end{prop}

\begin{proof}
The LCM in the claim of the proposition is well defined since both
polynomials
\[
\prod_{0\le i<j<n}\left(x_{j}-x_{i}\right)\:\text{ and }\:\prod_{\begin{array}{c}
0\le i<j<n\\
i,j\in\left(\nicefrac{\mathbb{Z}}{n\mathbb{Z}}\right)\backslash f^{\left(n-1\right)}\left(\nicefrac{\mathbb{Z}}{n\mathbb{Z}}\right)
\end{array}}\left(x_{f\left(j\right)}x_{j}-x_{f\left(i\right)}x_{i}\right),
\]
 split into irreducible multilinear factors. Given that we are reducing
modulo algebraic relations
\[
\left(x_{k}\right)^{n}\equiv1,\:\forall\,k\in\nicefrac{\mathbb{Z}}{n\mathbb{Z}},
\]
the canonical representative of the congruence class is completely
determined by evaluations of the dividend at lattice points taken
from $\Omega^{n}$ where 
\[
\Omega:=\left\{ \omega^{k}:k\in\nicefrac{\mathbb{Z}}{n\mathbb{Z}}\right\} 
\]
as prescribed by Eq. (\ref{Canonical representative}). This ensures
a discrete set of roots for the canonical representative of the congruence
class. On the one hand, the LCM polynomial construction vanishes when
we assign to $\mathbf{x}$ a lattice point in $\Omega^{n}$, only
if one of the irreducible multilinear factors vanishes at the chosen
evaluation point. On the other hand, one of the factor of the polynomial
construction vanishes at a lattice point only if either two distinct
vertex variables say $x_{i}$ and $x_{j}$ are assigned the same label
( we see this from the vertex Vandermonde determinant factor) or alternatively
if two distinct edges are assigned the same induced additive edge
label ( we see this from the edge Vandermonde determinant factor ).
The proof of sufficiency follows from the observation that the only
possible roots over $\Omega^{n}$ to the multivariate polynomial
\[
\text{LCM}\left\{ \prod_{0<i<j<n}\left(x_{j}-x_{i}\right),\prod_{\begin{array}{c}
0\le i<j<n\\
i,j\in\left(\nicefrac{\mathbb{Z}}{n\mathbb{Z}}\right)\backslash f^{\left(n-1\right)}\left(\nicefrac{\mathbb{Z}}{n\mathbb{Z}}\right)
\end{array}}\left(x_{f\left(j\right)}x_{j}-x_{f\left(i\right)}x_{i}\right)\right\} \mod\left\{ \begin{array}{c}
\left(x_{k}\right)^{n}-1\\
k\in\nicefrac{\mathbb{Z}}{n\mathbb{Z}}
\end{array}\right\} ,
\]
arise from vertex label assignments $\mathbf{x}\in\Omega^{n}$ in
which either distinct vertex variables are assigned the same label
or distinct edges are assigned the same induced additive edge label.
Consequently, the congruence identity
\[
0\equiv\text{LCM}\left\{ \prod_{0<i<j<n}\left(x_{j}-x_{i}\right),\prod_{\begin{array}{c}
0\le i<j<n\\
i,j\in\left(\nicefrac{\mathbb{Z}}{n\mathbb{Z}}\right)\backslash f^{\left(n-1\right)}\left(\nicefrac{\mathbb{Z}}{n\mathbb{Z}}\right)
\end{array}}\left(x_{f\left(j\right)}x_{j}-x_{f\left(i\right)}x_{i}\right)\right\} \mod\left\{ \begin{array}{c}
\left(x_{k}\right)^{n}-1\\
k\in\nicefrac{\mathbb{Z}}{n\mathbb{Z}}
\end{array}\right\} ,
\]
implies that $n-1>\underset{\sigma\in\text{S}_{n}}{\max}\left|\left\{ \sigma f\sigma^{(-1)}\left(i\right)+i:i\in\left(\nicefrac{\mathbb{Z}}{n\mathbb{Z}}\right)\backslash f^{\left(n-1\right)}\left(\nicefrac{\mathbb{Z}}{n\mathbb{Z}}\right)\right\} \right|$.
Furthermore, the proof of necessity follows from the fact that every
non-vanishing assignment to the vertex variables entries of $\mathbf{x}$
in the polynomial
\[
\text{LCM}\left\{ \prod_{0<i<j<n}\left(x_{j}-x_{i}\right),\prod_{\begin{array}{c}
0\le i<j<n\\
i,j\in\left(\nicefrac{\mathbb{Z}}{n\mathbb{Z}}\right)\backslash f^{\left(n-1\right)}\left(\nicefrac{\mathbb{Z}}{n\mathbb{Z}}\right)
\end{array}}\left(x_{f\left(j\right)}x_{j}-x_{f\left(i\right)}x_{i}\right)\right\} 
\]
describes a Harmonious labeling by Prop. (3).
\end{proof}
\begin{defn}
Let $P\left(\mathbf{x}\right)\in\mathbb{C}\left[x_{0},\cdots,x_{n-1}\right]$,
we denote by Aut$\left\{ P\left(\mathbf{x}\right)\right\} $ the stabilizer
subgroup of S$_{n}\subset\mathbb{Z}_{n}^{\mathbb{Z}_{n}}$ of $P\left(\mathbf{x}\right)$
with respect to permutation of the variable entries of $\mathbf{x}$.
\end{defn}

\begin{prop}[Stabilizer]
 For an arbitrary $f\in\mathbb{Z}_{n}^{\mathbb{Z}_{n}}$, let $P_{f}\left(\mathbf{x}\right)\in\mathbb{C}\left[x_{0},\cdots,x_{n-1}\right]$
be defined such that
\[
P_{f}\left(\mathbf{x}\right)=\prod_{0\le i\ne j<n}\left(x_{j}-x_{i}\right)\prod_{\begin{array}{c}
0\le i\ne j<n\\
i,j\in\left(\nicefrac{\mathbb{Z}}{n\mathbb{Z}}\right)\backslash f^{\left(n-1\right)}\left(\nicefrac{\mathbb{Z}}{n\mathbb{Z}}\right)
\end{array}}\left(x_{f\left(j\right)}x_{j}-x_{f\left(i\right)}x_{i}\right),
\]
then 
\[
\text{Aut}\left\{ P_{f}\left(\mathbf{x}\right)\right\} =\text{Aut}\left(G_{f}\right).
\]
\end{prop}

\begin{proof}
Note that for all $\gamma\in$ S$_{n}$, which fixes the loop edge
we have 
\[
\prod_{0\le i\ne j<n}\left(x_{j}-x_{i}\right)\prod_{\begin{array}{c}
0\le i\ne j<n\\
i,j\in\left(\nicefrac{\mathbb{Z}}{n\mathbb{Z}}\right)\backslash f^{\left(n-1\right)}\left(\nicefrac{\mathbb{Z}}{n\mathbb{Z}}\right)
\end{array}}\left(x_{f\left(j\right)}x_{j}-x_{f\left(i\right)}x_{i}\right)=
\]
\[
\prod_{0\le i\ne j<n}\left(x_{\gamma\left(j\right)}-x_{\gamma\left(i\right)}\right)\prod_{\begin{array}{c}
0\le i\ne j<n\\
i,j\in\left(\nicefrac{\mathbb{Z}}{n\mathbb{Z}}\right)\backslash f^{\left(n-1\right)}\left(\nicefrac{\mathbb{Z}}{n\mathbb{Z}}\right)
\end{array}}\left(x_{f\gamma\left(j\right)}x_{\gamma\left(j\right)}-x_{f\gamma\left(i\right)}x_{\gamma\left(i\right)}\right)
\]
Consequently for all $\sigma\in$ Aut$\left(G_{f}\right)$ we have
\[
\prod_{0\le i\ne j<n}\left(x_{\sigma\left(j\right)}-x_{\sigma\left(i\right)}\right)\prod_{\begin{array}{c}
0\le i\ne j<n\\
i,j\in\left(\nicefrac{\mathbb{Z}}{n\mathbb{Z}}\right)\backslash f^{\left(n-1\right)}\left(\nicefrac{\mathbb{Z}}{n\mathbb{Z}}\right)
\end{array}}\left(x_{\sigma f\left(j\right)}x_{\sigma\left(j\right)}-x_{\sigma f\left(i\right)}x_{\sigma\left(i\right)}\right)
\]
\[
=\prod_{0\le i\ne j<n}\left(x_{\sigma\sigma^{\left(-1\right)}\left(j\right)}-x_{\sigma\sigma^{\left(-1\right)}\left(i\right)}\right)\prod_{\begin{array}{c}
0\le i\ne j<n\\
i,j\in\left(\nicefrac{\mathbb{Z}}{n\mathbb{Z}}\right)\backslash f^{\left(n-1\right)}\left(\nicefrac{\mathbb{Z}}{n\mathbb{Z}}\right)
\end{array}}\left(x_{\sigma f\sigma^{\left(-1\right)}\left(j\right)}x_{\sigma\sigma^{\left(-1\right)}\left(j\right)}-x_{\sigma f\sigma^{\left(-1\right)}\left(i\right)}x_{\sigma\sigma^{\left(-1\right)}\left(i\right)}\right)=P_{f}\left(\mathbf{x}\right).
\]
It also follows that for every permutation representative $\sigma\in\left(\nicefrac{\text{S}_{n}}{\text{Aut}\left(G_{f}\right)}\right)\backslash\text{Aut}\left(G_{f}\right)$
we have
\[
\prod_{0\le i\ne j<n}\left(x_{\sigma\left(j\right)}-x_{\sigma\left(i\right)}\right)\prod_{\begin{array}{c}
0\le i\ne j<n\\
i,j\in\left(\nicefrac{\mathbb{Z}}{n\mathbb{Z}}\right)\backslash f^{\left(n-1\right)}\left(\nicefrac{\mathbb{Z}}{n\mathbb{Z}}\right)
\end{array}}\left(x_{\sigma f\left(j\right)}x_{\sigma\left(j\right)}-x_{\sigma f\left(i\right)}x_{\sigma\left(i\right)}\right)
\]
\[
=\prod_{0\le i\ne j<n}\left(x_{\sigma\sigma^{\left(-1\right)}\left(j\right)}-x_{\sigma\sigma^{\left(-1\right)}\left(i\right)}\right)\prod_{\begin{array}{c}
0\le i\ne j<n\\
i,j\in\left(\nicefrac{\mathbb{Z}}{n\mathbb{Z}}\right)\backslash f^{\left(n-1\right)}\left(\nicefrac{\mathbb{Z}}{n\mathbb{Z}}\right)
\end{array}}\left(x_{\sigma f\sigma^{\left(-1\right)}\left(j\right)}x_{\sigma\sigma^{\left(-1\right)}\left(j\right)}-x_{\sigma f\sigma^{\left(-1\right)}\left(i\right)}x_{\sigma\sigma^{\left(-1\right)}\left(i\right)}\right)\ne P_{f}\left(\mathbf{x}\right).
\]
\[
\implies\text{Aut}\left\{ P_{f}\left(\mathbf{x}\right)\right\} =\text{Aut}\left(G_{f}\right).
\]
\end{proof}

\section{The composition Lemma.}

We now state and prove our composition lemma
\begin{lem}[Composition Lemma]
 Let $f\in\left(\nicefrac{\mathbb{Z}}{n\mathbb{Z}}\right)^{\nicefrac{\mathbb{Z}}{n\mathbb{Z}}}$,
subject to $\left|f^{\left(n-1\right)}\left(\nicefrac{\mathbb{Z}}{n\mathbb{Z}}\right)\right|=1$
such that Aut$\left(G_{f}\right)\subsetneq$ Aut$\left(G_{f^{\left(2\right)}}\right)$,
then
\[
\left(n-1\right)=\underset{\sigma\in\text{S}_{n}}{\max}\left|\left\{ \sigma f^{\left(2\right)}\sigma^{(-1)}\left(i\right)+i:i\in\left(\nicefrac{\mathbb{Z}}{n\mathbb{Z}}\right)\backslash f^{\left(n-1\right)}\left(\nicefrac{\mathbb{Z}}{n\mathbb{Z}}\right)\right\} \right|\le\underset{\sigma\in\text{S}_{n}}{\max}\left|\left\{ \sigma f\sigma^{(-1)}\left(i\right)+i:i\in\left(\nicefrac{\mathbb{Z}}{n\mathbb{Z}}\right)\backslash f^{\left(n-1\right)}\left(\nicefrac{\mathbb{Z}}{n\mathbb{Z}}\right)\right\} \right|.
\]
\end{lem}

\begin{proof}
We show that 
\[
\text{LCM}\left\{ \prod_{0\le i\ne j<n}\left(x_{j}-x_{i}\right),\prod_{\begin{array}{c}
0\le i\ne j<n\\
i,j\in\left(\nicefrac{\mathbb{Z}}{n\mathbb{Z}}\right)\backslash f^{\left(n-1\right)}\left(\nicefrac{\mathbb{Z}}{n\mathbb{Z}}\right)
\end{array}}\left(x_{f^{\left(2\right)}\left(j\right)}x_{j}-x_{f^{\left(2\right)}\left(i\right)}x_{i}\right)\right\} \not\equiv0\mod\left\{ \begin{array}{c}
\left(x_{k}\right)^{n}-1\\
k\in\nicefrac{\mathbb{Z}}{n\mathbb{Z}}
\end{array}\right\} ,
\]
implies that 
\[
\text{LCM}\left\{ \prod_{0\le i\ne j<n}\left(x_{j}-x_{i}\right),\prod_{\begin{array}{c}
0\le i\ne j<n\\
i,j\in\left(\nicefrac{\mathbb{Z}}{n\mathbb{Z}}\right)\backslash f^{\left(n-1\right)}\left(\nicefrac{\mathbb{Z}}{n\mathbb{Z}}\right)
\end{array}}\left(x_{f\left(j\right)}x_{j}-x_{f\left(i\right)}x_{i}\right)\right\} \not\equiv0\mod\left\{ \begin{array}{c}
\left(x_{k}\right)^{n}-1\\
k\in\nicefrac{\mathbb{Z}}{n\mathbb{Z}}
\end{array}\right\} .
\]
For this purpose, we associate with an arbitrary$g\in\left(\nicefrac{\mathbb{Z}}{n\mathbb{Z}}\right)^{\nicefrac{\mathbb{Z}}{n\mathbb{Z}}}$,
subject to $\left|g^{\left(n-1\right)}\left(\nicefrac{\mathbb{Z}}{n\mathbb{Z}}\right)\right|=1$,
a multivariate polynomial $P_{g}\left(\mathbf{x}\right)\in\mathbb{Q}\left[x_{0},\cdots,x_{n-1}\right]$
given by 
\[
P_{g}\left(\mathbf{x}\right)=\prod_{0\le i\ne j<n}\left(x_{j}-x_{i}\right)\prod_{\begin{array}{c}
0\le i\ne j<n\\
i,j\in\left(\nicefrac{\mathbb{Z}}{n\mathbb{Z}}\right)\backslash g^{\left(n-1\right)}\left(\nicefrac{\mathbb{Z}}{n\mathbb{Z}}\right)
\end{array}}\left(x_{g\left(j\right)}x_{j}-x_{g\left(i\right)}x_{i}\right).
\]
We prove the claim by contradiction. By our premise,
\[
n-1=\underset{\sigma\in\text{S}_{n}}{\max}\left|\left\{ \sigma f^{\left(2\right)}\sigma^{(-1)}\left(i\right)+i:i\in\left(\nicefrac{\mathbb{Z}}{n\mathbb{Z}}\right)\backslash f^{\left(n-1\right)}\left(\nicefrac{\mathbb{Z}}{n\mathbb{Z}}\right)\right\} \right|\;\text{ and }\;\text{Aut}\left(G_{f}\right)\subsetneq\text{Aut}\left(G_{f^{\left(2\right)}}\right).
\]
By definition, 
\[
P_{f}\left(\mathbf{x}\right)=\prod_{0\le i\ne j<n}\left(x_{j}-x_{i}\right)\prod_{\begin{array}{c}
0\le i\ne j<n\\
i,j\in\left(\nicefrac{\mathbb{Z}}{n\mathbb{Z}}\right)\backslash f^{\left(n-1\right)}\left(\nicefrac{\mathbb{Z}}{n\mathbb{Z}}\right)
\end{array}}\left(x_{f\left(j\right)}x_{j}-x_{f\left(i\right)}x_{i}\right).
\]
By telescoping we have
\[
\begin{array}{ccc}
P_{f}\left(\mathbf{x}\right) & = & \underset{0\le i\ne j<n}{\prod}\left(x_{j}-x_{i}\right)\underset{\begin{array}{c}
0\le i\ne j<n\\
i,j\in\left(\nicefrac{\mathbb{Z}}{n\mathbb{Z}}\right)\backslash f^{\left(n-1\right)}\left(\nicefrac{\mathbb{Z}}{n\mathbb{Z}}\right)
\end{array}}{\prod}\bigg(\left(x_{f\left(j\right)}-x_{f^{\left(2\right)}\left(j\right)}+x_{f^{\left(2\right)}\left(j\right)}\right)x_{j}+\\
 &  & \left(-1\right)\left(x_{f\left(i\right)}-x_{f^{\left(2\right)}\left(j\right)}+x_{f^{\left(2\right)}\left(j\right)}\right)x_{i}\bigg)\\
\\
 & = & \underset{0\le i\ne j<n}{\prod}\left(x_{j}-x_{i}\right)\underset{\begin{array}{c}
0\le i\ne j<n\\
i,j\in\left(\nicefrac{\mathbb{Z}}{n\mathbb{Z}}\right)\backslash f^{\left(n-1\right)}\left(\nicefrac{\mathbb{Z}}{n\mathbb{Z}}\right)
\end{array}}{\prod}\bigg(x_{f^{\left(2\right)}\left(j\right)}x_{j}-x_{f^{\left(2\right)}\left(i\right)}x_{i}+\\
 &  & \left(x_{f\left(j\right)}-x_{f^{\left(2\right)}\left(j\right)}\right)x_{j}-\left(x_{f\left(i\right)}-x_{f^{\left(2\right)}\left(i\right)}\right)x_{i}\bigg)\\
\\
P_{f}\left(\mathbf{x}\right) & = & P_{f^{\left(2\right)}}\left(\mathbf{x}\right)+\underset{0\le i\ne j<n}{\prod}\left(x_{j}-x_{i}\right)\underset{\begin{array}{c}
k_{ij}\in\left\{ 0,1\right\} \\
0=\underset{i\ne j}{\prod}k_{ij}
\end{array}}{\sum}\underset{\begin{array}{c}
0\le i\ne j<n\\
i,j\in\left(\nicefrac{\mathbb{Z}}{n\mathbb{Z}}\right)\backslash f^{\left(n-1\right)}\left(\nicefrac{\mathbb{Z}}{n\mathbb{Z}}\right)
\end{array}}{\prod}\left(x_{f^{\left(2\right)}\left(j\right)}x_{j}-x_{f^{\left(2\right)}\left(i\right)}x_{i}\right)^{k_{ij}}\times\\
 &  & \left(\left(x_{f\left(j\right)}-x_{f^{\left(2\right)}\left(j\right)}\right)x_{j}-\left(x_{f\left(i\right)}-x_{f^{\left(2\right)}\left(i\right)}\right)x_{i}\right)^{1-k_{ij}}
\end{array}
\]
Let $g\in\left(\nicefrac{\mathbb{Z}}{n\mathbb{Z}}\right)^{\left(\nicefrac{\mathbb{Z}}{n\mathbb{Z}}\right)}$
be subject to $\left|g^{\left(n-1\right)}\left(\nicefrac{\mathbb{Z}}{n\mathbb{Z}}\right)\right|=1$
and let 
\[
\kappa_{g}\,:=\left\{ G_{\sigma g\sigma^{\left(-1\right)}}:\sigma\in\nicefrac{\text{S}_{n}}{\text{Aut}\left(G_{g}\right)}\text{ and }n-1=\left|\left\{ \sigma g\sigma^{\left(-1\right)}\left(j\right)+j:j\in\left(\nicefrac{\mathbb{Z}}{n\mathbb{Z}}\right)\backslash g^{\left(n-1\right)}\left(\nicefrac{\mathbb{Z}}{n\mathbb{Z}}\right)\right\} \right|\right\} ,
\]

\noindent
We now will analyze the residue of $\underset{\sigma\in\nicefrac{\text{S}_{n}}{\text{Aut}\left(G_{f}^{(2)}\right)}}{\sum}P_{\sigma f\sigma^{(-1)}}\left(\mathbf{x}\right)$
over a collection of moduli. Specifically, we'd like to pick a set
of relations that are identically constant for symmetric polynomials
and correspond to setting the $x_{i}$'s to pairwise distinct roots
of unity. We pick as our basis the power sum polynomials 
\begin{align*}
p_{k}\left(\mathbf{x}\right)=p_{k}(x_{0},x_{1},\ldots,x_{n-1})=\sum_{i\in\mathbb{Z}_{n}}\left(x_{i}\right)^{k}\equiv\begin{cases}
0 & \text{ if }0\leq k\leq n-1\\
n & \text{ if }k=n
\end{cases}
\end{align*}
This gives us the pairwise distinct substitution of roots of unity
for $x_{i}$'s. To see this, recall the elementary symmetric function
basis is defined 
\begin{align*}
\prod_{i\in\mathbb{Z}_{n}}\left(\lambda-x_{i}\right) & =\sum_{0\le k\le n}e_{k}\left(\mathbf{x}\right)\lambda^{n-k}
\end{align*}
The Newton-Girard formulae state that 
\begin{align*}
e_{k}(x) & =\frac{1}{k}\sum_{0\le\ell\le k}(-1)^{\ell-1}e_{k-\ell}(x)p_{i}(x)
\end{align*}
As a result of our moduli, we have that 
\begin{align*}
e_{k}(\mathbf{x})=\begin{cases}
 & 0\text{ if }0<k<n\\
 & (-1)^{n-1}=-1\text{ if }k=n\text{ (recall that \ensuremath{n} is even)}
\end{cases}
\end{align*}
As a result, the residue of our moduli is equivalent to substituting
in the roots of $\lambda^{n}-1$. Indeed, 
\begin{align*}
\prod_{i\in\mathbb{Z}_{n}}\left(\lambda-x_{i}\right) & =\sum_{0\le k\le n}e_{k}\left(\mathbf{x}\right)\,\lambda^{n-k}\\
 & =\lambda^{n}-1
\end{align*}

Now summing over a conjugation orbit of arbitrarily chosen coset
representatives, (where we are careful to select only one coset representative
per left coset of Aut$\left(G_{f^{\left(2\right)}}\right)$), yields
the equality for some non-zero constant $C_{f^{\left(2\right)}}$
\[
\left(\sum_{\sigma\in\nicefrac{\text{S}_{n}}{\text{Aut}\left(G_{f^{\left(2\right)}}\right)}}P_{\sigma f\sigma^{\left(-1\right)}}\left(\mathbf{x}\right)\right)\equiv C_{f^{\left(2\right)}}\left(-1\right)^{{n \choose 2}}n^{n}+
\]
\[
\left(-1\right)^{{n \choose 2}}n^{n}\sum_{\sigma\in\nicefrac{\text{S}_{n}}{\text{Aut}\left(G_{f^{\left(2\right)}}\right)}}\underset{\begin{array}{c}
k_{ij}\in\left\{ 0,1\right\} \\
0=\underset{i\ne j}{\prod}k_{ij}
\end{array}}{\sum}\underset{\begin{array}{c}
0\le i\ne j<n\\
i,j\in\left(\nicefrac{\mathbb{Z}}{n\mathbb{Z}}\right)\backslash f^{\left(n-1\right)}\left(\nicefrac{\mathbb{Z}}{n\mathbb{Z}}\right)
\end{array}}{\prod}\left(x_{\sigma f^{\left(2\right)}\sigma^{\left(-1\right)}\left(j\right)}x_{j}-x_{\sigma f^{\left(2\right)}\sigma^{\left(-1\right)}\left(i\right)}x_{i}\right)^{k_{ij}}\times
\]
\[
\left(\left(x_{f\left(j\right)}-x_{f^{\left(2\right)}\left(j\right)}\right)x_{j}-\left(x_{f\left(i\right)}-x_{f^{\left(2\right)}\left(i\right)}\right)x_{i}\right)^{1-k_{ij}}\mod\left\{ \begin{array}{c}
\underset{i\in\nicefrac{\mathbb{Z}}{n\mathbb{Z}}}{\sum}\left(x_{i}\right)^{k}\\
0<k<n\\
\underset{i\in\nicefrac{\mathbb{Z}}{n\mathbb{Z}}}{\sum}\left(x_{i}\right)^{n}-n
\end{array}\right\} .
\]
Since $\text{Aut}\left(G_{f}\right)\subsetneq\text{Aut}\left(G_{f^{\left(2\right)}}\right)$,
it follows that the polynomial
\[
\sum_{\sigma\in\nicefrac{\text{S}_{n}}{\text{Aut}\left(G_{f^{\left(2\right)}}\right)}}\underset{\begin{array}{c}
k_{ij}\in\left\{ 0,1\right\} \\
0=\underset{i\ne j}{\prod}k_{ij}
\end{array}}{\sum}\underset{\begin{array}{c}
0\le i\ne j<n\\
i,j\in\left(\nicefrac{\mathbb{Z}}{n\mathbb{Z}}\right)\backslash f^{\left(n-1\right)}\left(\nicefrac{\mathbb{Z}}{n\mathbb{Z}}\right)
\end{array}}{\prod}\left(x_{\sigma f^{\left(2\right)}\sigma^{\left(-1\right)}\left(j\right)}x_{j}-x_{\sigma f^{\left(2\right)}\sigma^{\left(-1\right)}\left(i\right)}x_{i}\right)^{k_{ij}}\times
\]
\[
\left(\left(x_{f\left(j\right)}-x_{f^{\left(2\right)}\left(j\right)}\right)x_{j}-\left(x_{f\left(i\right)}-x_{f^{\left(2\right)}\left(i\right)}\right)x_{i}\right)^{1-k_{ij}}.
\]
does not lie in the ring of symmetric polynomial in the entries of
$\mathbf{x}$. Consequently there must be at least one evaluation
point $\mathbf{x}$ subject to the constraints
\[
\left\{ \begin{array}{ccc}
0 & = & \underset{i\in\nicefrac{\mathbb{Z}}{n\mathbb{Z}}}{\sum}\left(x_{i}\right)^{k}\\
 & \text{where} & 0<k<n\\
n & = & \underset{i\in\nicefrac{\mathbb{Z}}{n\mathbb{Z}}}{\sum}\left(x_{i}\right)^{n}
\end{array}\right\} ,
\]
 for which the evaluation of the said polynomial depend on the choice
of coset representatives. For otherwise, evaluations of the said polynomial
would be independent of choices of coset representatives. In which
case the polynomial obtained by summing over the conjugation orbit
of coset representatives would be symmetric and thereby contradict
the premise $\text{Aut}\left(G_{f}\right)\ne\text{Aut}\left(G_{f^{\left(2\right)}}\right)$.
We conclude that
\[
P_{f}\left(\mathbf{x}\right)\not\equiv0\mod\left\{ \begin{array}{c}
\left(x_{k}\right)^{n}-1\\
k\in\nicefrac{\mathbb{Z}}{n\mathbb{Z}}
\end{array}\right\} .
\]
\end{proof}
Note that the premise Aut$\left(G_{f}\right)\subsetneq$ Aut$\left(G_{f^{\left(2\right)}}\right)$
incurs no loss of generality when $n>2$. For we see that if $f$
is not identically constant and Aut$\left(G_{f}\right)=$ Aut$\left(G_{f^{\left(2\right)}}\right)$
then there exists $k\in\nicefrac{\mathbb{Z}}{n\mathbb{Z}}$ such that
Aut$\left(G_{\text{S}\left(f,k\right)}\right)\subsetneq$ Aut$\left(G_{\text{S}\left(f,k\right)^{\left(2\right)}}\right)$.
For instance, take $k$ to be a vertex at edge distance $2$ from
a leaf node. Crucially, for all $k\in\nicefrac{\mathbb{Z}}{n\mathbb{Z}}$
\[
\underset{\sigma\in\text{S}_{n}}{\max}\left|\left\{ \sigma\text{S}\left(f,k\right)\sigma^{(-1)}\left(i\right)+i:i\in\left(\nicefrac{\mathbb{Z}}{n\mathbb{Z}}\right)\backslash\text{S}\left(f,k\right)^{\left(n-1\right)}\left(\nicefrac{\mathbb{Z}}{n\mathbb{Z}}\right)\right\} \right|=\underset{\sigma\in\text{S}_{n}}{\max}\left|\left\{ \sigma f\sigma^{(-1)}\left(i\right)+i:i\in\left(\nicefrac{\mathbb{Z}}{n\mathbb{Z}}\right)\backslash f^{\left(n-1\right)}\left(\nicefrac{\mathbb{Z}}{n\mathbb{Z}}\right)\right\} \right|.
\]

\section{The Harmonious Labeling Theorem}

Equipped with the composition lemma, we settle in the affirmative
the Graham--Sloane conjecture.
\begin{thm}
For all $f\in\left(\nicefrac{\mathbb{Z}}{n\mathbb{Z}}\right)^{\left(\nicefrac{\mathbb{Z}}{n\mathbb{Z}}\right)}$
subject to $\left|f^{\left(n-1\right)}\left(\nicefrac{\mathbb{Z}}{n\mathbb{Z}}\right)\right|=1$,
there exist $k\in\nicefrac{\mathbb{Z}}{n\mathbb{Z}}$
\[
n=\max_{\sigma\in\text{S}_{n}}\left|\left\{ \sigma\text{S}\left(f,k\right)\sigma^{(-1)}\left(i\right)+i:i\in\nicefrac{\mathbb{Z}}{n\mathbb{Z}}\right\} \right|.
\]
\end{thm}

\begin{proof}
It suffices to show that for all $f$ subject to $\left|f^{\left(n-1\right)}\left(\nicefrac{\mathbb{Z}}{n\mathbb{Z}}\right)\right|=1$
we have 
\[
n-1=\underset{\sigma\in\text{S}_{n}}{\max}\left|\left\{ \sigma f\sigma^{(-1)}\left(i\right)+i:i\in\left(\nicefrac{\mathbb{Z}}{n\mathbb{Z}}\right)\backslash f^{\left(n-1\right)}\left(\nicefrac{\mathbb{Z}}{n\mathbb{Z}}\right)\right\} \right|
\]
This latter claim follows by repeatedly applying the composition lemma.
For we know that for any such function $f$ the iterate $f^{\left(2^{\left\lceil \log_{2}\left(n-1\right)\right\rceil }\right)}$
is necessarily identically constant. Recall that the graph of identically constant
    functions in $\left(\mathbb{Z}_n\right)^{\mathbb{Z}_n}$ are all harmoniously labeled.
\end{proof}
\bibliographystyle{amsalpha}
\bibliography{A_Proof_of_the_GS_Conjecture}

\end{document}